\documentclass[twoside]{article}

\usepackage{amsmath,amsthm,amssymb,bm}

\usepackage{newtxtext,newtxmath}

\usepackage[text={12.5cm,19cm},centering,paperwidth=17cm,paperheight=24cm]{geometry}

\usepackage[fontsize=10.4pt]{scrextend}

\def\titlerunning#1{\gdef\titrun{#1}}

\makeatletter
\def\author#1{\gdef\autrun{\def\and{\unskip, }#1}\gdef\@author{#1}}
\def\address#1{{\def\and{\\\hspace*{15.6pt}}\renewcommand{\thefootnote}{}\footnote{#1}}\markboth{\autrun}{\titrun}}
\makeatother

\def\subjclass#1{\par\bigskip\noindent\textbf{Mathematics Subject Classification 2020.} #1}
\def\keywords#1{\par\smallskip\noindent\textbf{Keywords.} #1}

\newenvironment{dedication}{\itshape\center}{\par\medskip}
\newenvironment{acknowledgments}{\bigskip\small\noindent\textit{Acknowledgments.}}{\par}

\newtheorem{thm}{Theorem}

\theoremstyle{definition}

\newtheorem*{rem}{Remark}

\numberwithin{equation}{section}

\usepackage[hyperfootnotes=false,colorlinks=true,allcolors=blue]{hyperref}

\begin{document}

\titlerunning{Trace formulas}

\title{\textbf{Trace formulas for the modified Mathieu equation}}

\author{Leon A. Takhtajan}

\date{}

\maketitle

\address{Leon  A. Takhtajan:  leontak@math.stonybrook.edu\and
Department of Mathematics, Stony Brook University, Stony Brook, NY 11794-3651, USA \and
Euler International Mathematical Institute, Pesochnaya Nab. 10, Saint Petersburg 197022, Russia}

\begin{dedication}
To my friend Ari Laptev on the occasion of his 70th birthday
\end{dedication}

\begin{abstract}
For the radial and one-dimensional Schr\"{o}dinger operator $H$ with growing potential $q(x)$ we outline a method of
obtaining the trace identities --- an asymptotic expansion of the Fredholm determinant $\mathrm{det}_{F}(H-\lambda I)$ as $\lambda\to-\infty$.  As an illustrating example, we consider
Schr\"{o}dinger operator with the potential $q(x)=2\cosh 2x$,  associated with the modified Mathieu equation.
\end{abstract}

\newcommand{\re}{\mathrm{Re}}
\newcommand{\tr}{\mathrm{Tr}}
% differentials
\newcommand{\del}{\partial}
\newcommand{\delb}{\bar\partial}
\newcommand{\delp}{\partial^\prime}
\newcommand{\delpp}{\partial^{\prime\prime}}
\newcommand{\deltap}{\d}
\newcommand{\deltapp}{\delta}
\newcommand{\eqdef}{\overset{\text{def}}{=}}
\newcommand{\ptheta}{\check{\theta}}
\newcommand{\vk}{\varkappa}
\newcommand{\var}{\boldsymbol{\delta}}
\newcommand{\T}{\mathfrak{T}}
\newcommand{\Pro}{\mathfrak{P}}
\newcommand{\D}{\mathfrak{D}}
\newcommand{\ihalf}{\frac{i}{2}}
\newcommand{\qc}{\mathcal{Q}\mathcal{C}(\Gamma,\tilde\Gamma)}
\newcommand{\qsubgp}{\stackrel{\Q}{<}}
\newcommand{\RR}{{\mathbb{R}}}
\newcommand{\NN}{{\mathbb{N}}}
\newcommand{\ZZ}{{\mathbb{Z}}}
\newcommand{\Q}{{\mathbb{Q}}}
\newcommand{\C}{{\mathbb{C}}}
\newcommand{\al}{\alpha}
\newcommand{\be}{\beta}
\newcommand{\de}{\delta}
\newcommand{\g}{\gamma}
\newcommand{\s}{\gamma}
\newcommand{\ka}{\kappa}
\newcommand{\epi}{<_{e}}
\newcommand{\vp}{\varphi}
\newcommand{\N}{\mathbb{N}}
\newcommand{\cl}{\overline}
\newcommand{\bk}{\backslash}
\newcommand{\Ga}{\Gamma}
\newcommand{\Gaf}{\Gamma_F}
\newcommand{\Gas}{\Gamma_S}
\newcommand{\Gr}{\mathrm{SL}(2,\mathbb{R})}
\newcommand{\Gc}{\mathrm{SL}(2,\mathbb{C})}
\newcommand{\Gu}{\mathrm{SU}(1,1)}
\newcommand{\pa}{\partial}
\newcommand{\la}{\langle}
\newcommand{\ra}{\rangle}
\newcommand{\bpa}{\overline{\partial}}
\newcommand{\Ka}{K\"{a}hler\:}
\newcommand{\Te}{Teichm\"{u}ller\:}
\newcommand{\Ac}{S_{\Gamma}\,}
\newcommand{\bu}{\bullet}
\newcommand{\ov}{\overline}
\newcommand{\ep}{\epsilon}
\newcommand{\vep}{\varepsilon}
\newcommand{\gr}{\mathrm{\mathfrak{sl}}(2,\mathbb{R})}
\newcommand{\z}{\bar{z}}
\newcommand{\op}{e^{\frac{t}{2}H}}
\newcommand{\ac}{\int_{X} \phi_{z}\phi_{\cl{z}}+e^{\phi}d^2z}
\newcommand{\ma}[4]{(\begin{smallmatrix}
              #1 & #2 \\ #3 & #4
             \end{smallmatrix})}
\newcommand{\intf}{\int_{-\infty}^{\infty}}
\newcommand{\curly}[1]{\mathscr{#1}}
\newcommand{\cA}{\curly{A}}
\newcommand{\cB}{\curly{B}}
\newcommand{\cC}{\curly{C}}
\newcommand{\cD}{\curly{D}}
\newcommand{\cE}{\curly{E}}
\newcommand{\cF}{\curly{F}}
\newcommand{\cG}{\curly{G}}
\newcommand{\cH}{\curly{H}}
\newcommand{\cI}{\curly{I}}
\newcommand{\cJ}{\curly{J}}
\newcommand{\cK}{\curly{K}}
\newcommand{\cL}{\curly{L}}
\newcommand{\cM}{\curly{M}}
\newcommand{\cN}{\curly{N}}
\newcommand{\cO}{\curly{O}}
\newcommand{\cP}{\curly{P}}
\newcommand{\cQ}{\curly{Q}}
\newcommand{\cR}{\curly{R}}
\newcommand{\cS}{\curly{S}}
\newcommand{\cT}{\curly{T}}
\newcommand{\cU}{\curly{U}}
\newcommand{\cV}{\curly{V}}
\newcommand{\cW}{\curly{W}}
\newcommand{\cX}{\curly{X}}
\newcommand{\cY}{\curly{Y}}
\newcommand{\cZ}{\curly{Z}}

\section{Introduction} Mathieu functions, solutions of the Mathieu equation and modified Mathieu equation, have many applications in mathematics and theoretical physics. In particular, the modified Mathieu differential equation
$$\frac{d^{2}\psi}{dx^{2}}+(\lambda-2\cosh 2x)\psi=0$$
on the real axis $ -\infty<x<\infty$, is a Schr\"{o}dinger equation 
$$H\psi=\lambda\psi$$ 
with the Hamiltionian 
\begin{equation}\label{H-M}
H=-\frac{d^{2}}{dx^{2}} +2\cosh 2x
\end{equation}
in the Hilbert space $L^{2}(\RR)$.

A positive, self-adjoint operator $H$ given by \eqref{H-M} is a Hamiltonian of the quantum two-particle periodic Toda chain (after separation of the center of mass), and has a purely discrete spectrum. As its classical analog, the quantum $N$-particle periodic Toda chain is also completely integrable, as was shown by M. Gutzwiller \cite{G} for $N=2,3$ and by E.K. Sklyanin \cite{Sklyanin} for general $N$. For $N=2$ the exact quantization condition in \cite{G} was obtained by applying Floquet theory to the potential $q(x)=2\cosh 2x$ with pure imaginary period, and is formulated in terms of infinite Hill determinants (see \cite[Ch. XIX]{WW}). Sklyanin's method of quantum separation of variables was developed further by Q. Pasquier and M. Gaudin \cite{PG}, and by S. Kharchev and D. Lebedev \cite{KL}.

Recently, N.A. Nekrasov and S.L. Shatashvili \cite{NS} found a remarkable connection between quantum integrable systems and $\mathcal{N}=2$ four-dimensional supersymmetric Yang-Mills theories in the special $\Omega$-background. Among other interesting results they showed, at the physical level of rigor, that exact quantization conditions for the quantum $N$-particle periodic Toda chain can be written in terms of the so-called effective twisted superpotential for the low energy effective two-dimensional gauge theory. In particular, Nekrasov and Shatashvili gave a new interpretation of Gutzwiller and Kharchev and Lebedev quantization conditions.

Replacing the kinetic term $\bm{P}^{2}$ in the Hamiltonian $H$ by $2\cosh \bm{P}$, where $\bm{P}=-i\dfrac{d}{dx}$ is the quantum-mechanical momentum operator, one gets the following functional-difference operator on $L^{2}(\RR)$,
\begin{equation}\label{H-mirror}
\hat{H}=2\cosh \bm{P}+2\cosh 2\bm{Q},
\end{equation}
where $\bm{Q}$ is the quantum-mechanical position operator.
The operator $\hat{H}$ is a positive self-adjoint operator on $L^{2}(\RR)$ with a purely discrete spectrum, and $\hat{H}^{-1}$ is a trace class operator.
Operators of such type naturally appear in the theory of topological strings as a quantization of mirror curves of non-compact Calabi-Yau threefolds, called ``topological strings/spectral theory (TS/ST) duality''. These operators have a purely discrete spectrum, and their Fredholm determinants 
\begin{equation}\label{F-det}
\mathrm{det}_{F}(\hat{H}-\lambda I)=\frac{\det(\hat{H}-\lambda I)}{\det\hat{H}}=\prod_{n=1}^{\infty}\left(1-\frac{\lambda}{\lambda_{n}}\right),
\end{equation}
where $\lambda_{n}$ are the eigenvalues,
are conjectured to encode deep relations in the enumerative geometry of the corresponding Calabi-Yau manifolds. We refer the reader to M. Mari\~{n}o comprehensive survey \cite{Mar} for the precise description of the method, formulation of the main conjectures and references.  Developing this method, A. Grassi, J. Gu and M. Mari\~{n}o  in \cite{Mar2} conjectured an exact formula for the Fredholm determinant of the modified Mathieu operator in terms of the so-called ``exact quantum periods''.
%r as a consequence of our conjectures for toric Calabi-Yaus. 
Though in certain special cases TS/ST conjectures can be confirmed by numerical computations, their mathematical derivation seems to be out of reach. 

Nevertheless, it is quite remarkable that operators of the type 
\eqref{H-mirror} can be rigorously studied by the method developed by Ari Laptev more then twenty years ago in \cite{Laptev}. Namely, it was proved in \cite{LST} (see also \cite{LST-2}) that these functional-difference operators have a purely discrete spectrum and their eigenvalue counting function satisfies the Weyl's law.

The logarithmic derivative of the Fredholm determinant $a(\lambda)=\mathrm{det}_{F}(\hat{H}-\lambda I)$ is
$$-\frac{d}{d\lambda}\log a(\lambda)=\tr(\hat{H}-\lambda I)^{-1}=\sum_{n=1}^{\infty}\frac{1}{\lambda-\lambda_{n}}.$$
As $\lambda\to-\infty$, it admits the asymptotic expansion
$$\frac{d}{d\lambda}\log a(\lambda)=\alpha_{0}(\nu)+\sum_{n=1}^{\infty}\frac{c_{n}}{\nu^{n}}+O(|\nu|^{-\infty}),$$
where $\nu=\sqrt{-\lambda}$ and the function $\al_{0}(\nu)$ is determined by the asymptotic of the eigenvalues 
(cf. \cite[Theorem 3.1]{D}). The coefficients $c_{2k}$ represent regularized divergent series $\sum_{n=1}^{\infty}\lambda_{n}^{k}$ and it is expected that they are expressed in terms of the potential of the functional-difference operator $\hat{H}$. Relations of such type are called \emph{trace identities}.

The trace identities for the Sturm-Liouville operators were derived in the classic papers by I.M. Gelfand and B.M. Levitan \cite{G-L} and L.A. Dikii \cite{D},  for the radial Schr\"{o}dinger operator with rapidly decaying potential $q(x)$ --- by V.S. Buslaev and L.D. Faddeev in \cite{B-F}, and for the one-dimensional Schr\"{o}dinger operator with rapidly decaying potential $q(x)$ --- by L.D. Faddeev and V.E. Zakharov in \cite{F-Z} (see also \cite{LT}).  

Though the trace identities for growing as $|x|\to\infty$  potential $q(x)$ have not been explicitly considered in the literature, a closed topic --- semiclassical expansion in a complex domain, illustrated by the case of the homogeneous quartic harmonic oscillator --- was investigated by A. Voros in \cite{Voros1, Voros2}.
%To the best of our knowledge, the trace identities for growing as $|x|\to\infty$  potential $q(x)$ have not been considered in the literature. Instead one studies  
One should also mention the relative trace identities --- the asymptotic
expansion of
$$\tr\left\{\left(-\frac{d^{2}}{dx^{2}}+q(x)+r(x)-\lambda I\right)^{-1}-\left(-\frac{d^{2}}{dx^{2}}+q(x)-\lambda I\right)^{-1}\right\}$$
as $\lambda\to-\infty$, where potential $q(x)$ grows as $|x|\to\infty$,  and function $r(x)$ is rapidly decaying. The interesting example  $q(x)=x^{2}$ was considered by A. Pushnitski and I. Sorrell in \cite{P}.
%A. Voros papers  \cite{Voros1, Voros2} were devoted to a similar topic, semiclassical expansion in a complex domain, illustrated by the case of homogeneous
%quartic harmonic oscillator.

In the present paper we outline a simple method for the derivation of the trace identities for the Schr\"{o}dinger operator with the potential $q(x)$ that rapidly grows to infinity as $|x|\to\infty$, and illustrate it by explicit formulas for the case $q(x)=2\cosh 2x$, the modified Mathieu operator \eqref{H-M}.  In Sect. \ref{LG} we briefly review the Liouville-Green method, which is used to obtain the trace identities for the radial Schr\"{o}dinger operator  in Sect. \ref{half}, and for the one-dimensional Schr\"{o}dinger operator in Sect. \ref{line}. In Section \ref{E} we illustrate our method on two examples: a well-known radial Schr\"{o}dinger operator with potential $q(x)=e^{2x}$ in Sect. \ref{half-line}, and modified Mathieu operator \eqref{H-M} in Sect. \ref{cosh-line-0}. Our main result --- the trace identities for the modified Mathieu operator --- is presented in Theorem \ref{main-thm}. Finally, in Section \ref{E-value} we prove that modified Mathieu differential equation has a solution that behaves like the modified Bessel function of the second kind as $x\to\infty$.

\section{General case}
Here we consider the differential equation 
\begin{equation}\label{e-value-1}
-\psi''+ q(x)\psi=\lambda\psi
\end{equation}
one the half-line $0<x<\infty$ and on the real line $\infty<x<\infty$, where $q(x)$ is a positive smooth function that grows to infinity as $x\to\infty$ (or $|x|\to\infty$).
\subsection{The Liouville-Green method} \label{LG}

When $\lambda<0$, 
decaying as $x\to\infty$ solution of \eqref{e-value-1} has the following double asymptotic
\begin{equation}\label{chi-1}
\psi_{1}(x,\lambda)=\frac{C_{1}(\lambda)}{\sqrt[4]{q(x)-\lambda}}\,e^{-\int_{0}^{x}\sqrt{q(s)-\lambda}\,ds}(1+\vep_{1}(x,\lambda)),
\end{equation}
where in may cases
$$|\vep_{1}(x,\lambda)|\leq\frac{\tau(x)}{|\lambda|},\quad \tau(x)\to 0\quad\text{as}\quad x\to\infty.$$
Similarly, solution of \eqref{e-value-1} that grows as $x\to\infty$ has the asymptotic
\begin{equation}\label{chi-2}
\psi_{2}(x,\lambda)=\frac{C_{2}(\lambda)}{\sqrt[4]{q(x)-\lambda}}\,e^{\int_{0}^{x}\sqrt{q(s)-\lambda}\,ds}(1+\vep_{2}(x,\lambda)).
\end{equation}
This follows from the classical Liouville-Green method, developed by F. Olver \cite{Olver} (see also \cite{Fed}).

Put $Q(x,\lambda)=q(x)-\lambda$ and introduce
\begin{equation}\label{chi-3}
\chi(x,\lambda)=\log\psi_{1}(x,\lambda)+\frac{1}{4}\log Q(x,\lambda)+\int_{0}^{x}\sqrt{Q(s,\lambda)}\,ds-\log C_{1}(\lambda)
\end{equation}
and
\begin{equation}\label{sigma-1}
\sigma(x,\lambda)=\chi'(x,\lambda)=\frac{\psi_{1}'(x,\lambda)}{\psi_{1}(x,\lambda)}+\frac{Q'(x,\lambda)}{4Q(x,\lambda)}+\sqrt{Q(x,\lambda)},
\end{equation}
where prime stands for the $x$-derivative. Denoting for brevity $\sigma=\sigma(x,\lambda)$ and $Q=Q(x,\lambda)$, we see that the
function $\sigma$ satisfies the Riccati equation
\begin{align*}
\sigma' &=\frac{\psi_{1}''}{\psi_{1}}-\left(\frac{\psi_{1}'}{\psi_{1}}\right)^{2}+ \frac{Q''}{4Q}-\frac{1}{4}\left(\frac{Q'}{Q}\right)^{2} +\frac{Q'}{2\sqrt{Q}}\\
&=Q-\left(\sigma- \frac{Q'}{4Q}-\sqrt{Q}\right)^{2}+\frac{Q''}{4Q}-\frac{1}{4}\left(\frac{Q'}{Q}\right)^{2} +\frac{Q'}{2\sqrt{Q}},
\end{align*}
or 
\begin{equation}\label{Ric}
\sigma'=-\sigma^{2} +2\sqrt{Q}\sigma+\frac{Q'}{2Q}\sigma + \frac{Q''}{4Q}-\frac{5}{16}\left(\frac{Q'}{Q}\right)^{2}.
\end{equation}
\subsection{Trace identities}
The Riccati equation \eqref{Ric} can be used to obtain an asymptotic expansion of $\sigma(x,\lambda)$ as $\lambda\to-\infty$ and get the trace identities
for the Schr\"{o}dinger operator with growing as $x\to\infty$ potential $q(x)$
on $L^{2}(0,\infty)$ and on $L^{2}(\RR)$. In the latter case potential $q(x)$ is assumed to be an even function.

\subsubsection{Radial Schr\"{o}dinger operator}\label{half}
Here we consider the Schr\"{o}dinger operator $H$ on $L^{2}(0,\infty)$ with a smooth positive potential $q(x)$ satisfying 
\begin{equation} \label{q-I}
\int_{0}^{\infty}\frac{dx}{\sqrt{q(x)}}<\infty,
\end{equation}
supplemented with the boundary condition $\psi(0)=0$. Restricting further, we consider potentials satisfying the inequality $q(x)\geq Cx^{2+\vep}$ for some $\vep>0$. It follows from the Weyl's law that $H^{-1}$ is a trace class operator,
and its Fredholm determinant is an entire function of order $1/2$.

For every $\lambda$ equation  \eqref{e-value-1} for $0<x<\infty$ has a unique solution $\psi(x,\lambda)$ with the asymptotic 
\begin{equation}\label{x-as}
\psi(x,\lambda)=\frac{1}{\sqrt[4]{q(x)}}\,e^{-\int_{0}^{x}\sqrt{q(s)}\,ds}(1+o(1))\quad\text{as}\quad x\to\infty.
\end{equation}
For fixed $x$ solution $\psi(x,\lambda)$ is an entire function of $\lambda$ of order $1/2$ and $\psi(0,\lambda)$ is proportional to the Fredholm determinant
$a(\lambda)$ of the operator $H$,
$$a(\lambda)=\frac{\psi(0,\lambda)}{\psi(0,0)}.$$
Comparing equations \eqref{chi-1} and \eqref{x-as} for $\lambda<0$, we see that for such $\lambda$ solution $\psi(x,\lambda)$ has asymptotic \eqref{chi-1} with
$$C_{1}(\lambda)=e^{\int_{0}^{\infty}\big(\sqrt{Q(x,\lambda)}-\sqrt{q(x)}\big)dx}.$$
It follows from \eqref{q-I} that the integral in this formula is convergent, so $C_{1}(\lambda)$ is well-defined.

Putting in formula \eqref{chi-3} $\lambda =-\nu^{2}$, we have
\begin{gather*}
\chi(x,\lambda)=\log\psi(x,\lambda)+\frac{1}{4}\log(q(x)+\nu^{2})+\int_{0}^{x}\sqrt{q(s)+\nu^{2}}\,ds\\
-\int_{0}^{\infty}\Big(\sqrt{q(s)+\nu^{2}}-\sqrt{q(s)}\Big)ds.
\end{gather*}
From here we obtain
\begin{align*}
\chi(0,\lambda) & =\log a(\lambda) -a_{0}(\lambda),\\
\lim_{x\to\infty}\chi(x,\lambda) &=0,
\end{align*}
where
$$a_{0}(\lambda)= -\frac{1}{4}\log(q(0)+\nu^{2})+\int_{0}^{\infty}\Big(\sqrt{q(x)+\nu^{2}}-\sqrt{q(x)}\Big)dx -\log\psi(0,0)$$
(cf. \cite[Appendix A]{Voros2}).
It is not difficult to show that $a_{0}(\lambda)$ 
admits an asymptotic expansion as $\lambda=-\nu^{2}\to\infty$,
\begin{equation}\label{alpha-0-0}
a_{0}(\lambda)=\alpha(\nu)+\sum_{n=1}^{\infty}\frac{\al_{n}}{\nu^{n}} + O(|\nu|^{-\infty}),
\end{equation}
where the leading term --- a function $\al(\nu)$ --- is determined by the asymptotic of the eigenvalues and can be computed explicitly (see examples in Section \ref{E}). 

The function $\sigma(x,\lambda)=\chi^{\prime}(x,\lambda)$ satisfies the Riccati equation \eqref{Ric}, which admits
the asymptotic solution
\begin{equation}\label{c-series}
\sigma(x,\lambda)=\sum_{n=1}^{\infty}\frac{c_{n}(x)}{\nu^{n}}+ O(|\nu|^{-\infty}),
\end{equation}
where the coefficients $c_{n}(x)$ are determined recursively in terms of the potential $q(x)$ and its derivatives. Since
\begin{equation}\label{sigma-chi}
\log\chi(0,\lambda)=-\int_{0}^{\infty}\sigma(x,\lambda)dx,
\end{equation}
we have
\begin{equation}\label{a-asym-2}
\log a(\lambda)=\alpha_{0}(\lambda) 
-\int_{0}^{\infty}\sigma(x,\lambda)dx.
\end{equation}
Using the asymptotic expansion \eqref{c-series},
we obtain the trace identities
\begin{equation}\label{trace-1}
\log a(\lambda)=\alpha(\nu)+\sum_{n=1}^{\infty}\frac{c_{n}}{\nu^{n}}+O(|\nu|^{-\infty}),\quad\text{where}\quad c_{n}=\alpha_{n}-\int_{0}^{\infty}c_{n}(x)dx.
\end{equation}

\subsubsection{One-dimensional Schr\"{o}dinger operator} \label{line} Here we assume that the potential $q(x)$ is a smooth even function satisfying \eqref{q-I}. A fundamental system of solutions of the differential equation \eqref{e-value-1} is given by solutions $\psi_{1}(x,\lambda)$ and $\psi_{2}(x,\lambda)$ with the following asymptotic as $x\to\infty$,
\begin{align}
\psi_{1}(x,\lambda) & =\frac{1}{\sqrt[4]{q(x)}}\,e^{-\int_{0}^{x}\sqrt{q(s)}\,ds}(1+o(1)),\label{x-as-1}\\
\psi_{2}(x,\lambda) & =\frac{1}{\sqrt[4]{q(x)}}\,e^{\int_{0}^{x}\sqrt{q(s)}\,ds}(1+o(1)). \label{x-as-2}
\end{align}
Assuming that 
$$\lim_{x\to\infty}\frac{q'(x)}{\big(\sqrt{q(x)}\big)^{3}}=0,$$
we get from \eqref{x-as-1} and \eqref{x-as-2} that
$$W(\psi_{1},\psi_{2})=2,$$
where $W(f,g)=fg'-f'g$ is the Wronskian of two functions. The functions $\psi_{1}(x,\lambda)$ and $\psi_{2}(x,\lambda)$ satisfy asymptotic formulas \eqref{chi-1}--\eqref{chi-2}, where
$$C_{1}(\lambda)=e^{\int_{0}^{\infty}
\big(\sqrt{Q(x,\lambda)}-\sqrt{q(x)}\big)dx}\quad\text{and}\quad C_{2}(\lambda)=e^{-\int_{0}^{\infty}\big(\sqrt{Q(x,\lambda)}-\sqrt{q(x)}
\big)dx}.$$

Another fundamental system of solutions is given by the functions $\psi_{1}(-x,\lambda)$ and $\psi_{2}(-x,\lambda)$, and we have
\begin{equation}\label{-+-0}
\psi_{1}(x,\lambda)=t_{11}(\lambda)\psi_{1}(-x,\lambda)+t_{12}(\lambda)\psi_{2}(-x,\lambda),
\end{equation}
where
\begin{equation}\label{trans}
t_{12}(\lambda)=\frac{1}{2}W(\psi_{1}(x,\lambda),\psi_{1}(-x,\lambda)).
\end{equation}
For fixed $x$ the functions $\psi_{1}(x,\lambda)$ and $\psi_{2}(x,\lambda)$ are entire functions of order $1/2$, as it can be shown using condition \eqref{q-I}. Therefore, $t_{12}(\lambda)$ is an entire function
of order $1/2$ with zeros at the eigenvalues of the operator $H$. As in the previous section, for the Fredholm determinant $a(\lambda)$ of the operator $H$ we obtain
$$a(\lambda)=\frac{t_{12}(\lambda)}{t_{12}(0)}.$$ 

It follows from \eqref{chi-1} that the function  $\chi(x,\lambda)$, defined by \eqref{chi-3}, satisfies
$$\lim_{x\to\infty}\chi(x,\lambda)=0.$$
To investigate its behavior as $x\to-\infty$ we observe that it follows from \eqref{x-as-1}--\eqref{x-as-2} that as $x\to-\infty$ the first term in  \eqref{-+-0} is exponentially small with respect to the second term, so
$$\lim_{x\to-\infty}(\log\psi_{1}(x,\lambda)-\log\psi_{2}(-x,\lambda))=\log t_{12}(\lambda).$$
Since the antiderivative $\int_{0}^{x}\sqrt{Q(s,\lambda)}ds$ is an odd function of $x$, we obtain
$$\lim_{x\to-\infty}\chi(x,\lambda)=\log a(\lambda)-a_{0}(\lambda),$$
where now
$$a_{0}(\lambda)= \int_{-\infty}^{\infty}\Big(\sqrt{q(x)+\nu^{2}}-\sqrt{q(x)}\Big)dx -\log t_{12}(0).$$
As in Sect. \ref{half}, the function  $a_{0}(\lambda)$ admits an asymptotic expansion of the type \eqref{alpha-0-0}. The corresponding Riccati equation \eqref{Ric}
has an asymptotic solution 
$$\sigma(x,\lambda)=\sum_{n=1}^{\infty}\frac{c_{n}(x)}{\nu^{n}}+ O(|\nu|^{-\infty}),$$
where the coefficients $c_{n}(x)$ are determined recursively and are expressed in terms of the potential $q(x)$ and its derivatives.	
Thus we obtain the trace identities
\begin{equation}\label{trace-2}
\log a(\lambda)=\alpha(\nu)+\sum_{n=1}^{\infty}\frac{c_{n}}{\nu^{n}} +O(|\nu|^{-\infty}),\quad\text{where}\quad
c_{n}=\alpha_{n}-\int_{-\infty}^{\infty}c_{n}(x)dx.
\end{equation}
\begin{rem} The interesting example of the harmonic oscillator --- the potential $q(x)=x^{2}$ --- does not satisfy condition \eqref{q-I}. However, one can modify the outlined here method by exploiting the substitution $x=\sqrt{-\lambda}t$.\footnote{This is the substitution used in \cite{Olver} to obtain uniform asymptotic for Weber functions (parabolic cylinder functions).} In this case the Fredholm determinant is
$$a(\lambda)=\frac{2^{-\frac{\lambda}{2}}\sqrt\pi}{\Gamma(\frac{1-\lambda}{2})}$$
and corresponding trace identities give an alternate derivation of the Stirling asymptotic expansion for the Euler gamma-function (for  an equivalent approach, see \cite[Part IV]{Voros1}).
\end{rem} 
\section{Examples} \label{E}
\subsection{Potential $q(x)=e^{2x}$ on the half-line} \label{half-line}

It is well-known that the differential equation 
$$-\psi''+e^{2x}\psi=\lambda\psi$$
has two linearly independent solutions,  the modified Bessel function of the first kind $I_{\nu}(e^{x})$, and the modified Bessel function of the second kind $K_{\nu}(e^{x})$, where $\nu=\sqrt{-\lambda}$. 
For fixed $\nu$ and real $y\to\infty$ we have the asymptotic
\begin{equation} \label{K-z}
K_{\nu}(y)=\sqrt{\frac{\pi}{2y}}e^{-y}\left(1+O(y^{-1})\right),
\end{equation}
so the eigenvalues determined by the boundary condition $\psi(0)=0$ are the zeros $\{\lambda_{n}\}_{n=1}^{\infty}$ of an entire function $K_{\nu}(1)$, where $\lambda=-\nu^{2}$. The eigenvalues $\lambda_{n}$ are positive, simple and  accumulate to infinity, and 
\begin{equation}\label{a-D}
a(\lambda)=\frac{K_{\nu}(1)}{K_{0}(1)}.
\end{equation} 
\begin{rem} It is well-known (see, e.g., \cite{Polya}) that the total number of zeros of $K_{ik}(1)$ with $0<k<T$ is
$$\frac{T}{\pi}\log\frac{2T}{e}+O(1),$$
so for the eigenvalue counting function $N(\lambda)$ we obtain
$$N(\lambda)=\frac{\sqrt{\lambda}}{\pi}\log\frac{2\sqrt\lambda}{e}+O(1).$$
This also easily follows from the Weyl's law. Thus as $n\to\infty$ for the $n$-th eigenvalue $\lambda_{n}$ we have
$$\lambda_{n}\simeq\left(\frac{2\pi n}{eW(\frac{2\pi n}{e})}\right)^{2},$$
where $W(x)$ is the Lambert function, $W(x)e^{W(x)}=x$ for $x>0$. 
\end{rem}
 The asymptotic expansion of $\log a(\lambda)$ as $\lambda=-\nu^{2}\to-\infty$ can be easily obtained from the well-known asymptotic of the modified Bessel function of the second kind\footnote{One needs to put $z=1/\nu$ and $p=\nu/\sqrt{1+\nu^{2}}$ in formula (7.17) in \cite[Ch.10, \S7]{Olver}.}. However, it can also be derived directly, using the method outlined in Sect. \ref{half}.
 
Namely,  
we have
\begin{gather*}
\int_{0}^{x}\sqrt{e^{2s}+\nu^{2}}\,ds=\sqrt{\nu^{2}+e^{2x}}+\nu x-\nu\log\left(\nu+ \sqrt{\nu^{2}+e^{2x}}\right)\\ -\sqrt{\nu^{2}+1}+\nu\log\left(\nu+ \sqrt{\nu^{2}+1}\right).
\end{gather*}
Since
\begin{gather*}
\sqrt{\nu^{2}+e^{2x}}+\nu x-\nu\log\left(\nu+ \sqrt{\nu^{2}+e^{2x}}\right) = e^{x} + O(e^{-x})\quad\text{as}\quad x\to\infty,
\end{gather*}
comparison with \eqref{chi-1} gives $K_{\nu}(e^{x})=\psi_{1}(x,\lambda)$,
where
$$C_{1}(\lambda)=\sqrt{\frac{\pi}{2}}\,e^{\sqrt{\nu^{2}+1}-\nu\log\left(\nu- \sqrt{\nu^{2}+1}\right)}.$$
Of course, this is a well-known asymptotic of the modified Bessel function of the second kind,
\begin{equation} \label{K-unif}
K_{\nu}(e^{x})=\sqrt{\frac{\pi}{2}}\frac{e^{-\{\sqrt{\nu^{2}+e^{2x}}+\nu x-\nu\log(\nu+\sqrt{\nu^{2}+e^{2x}})\}}}{\sqrt[4]{\nu^{2}+e^{2x}}}(1+O(|\nu|^{-1})),
\end{equation}
which is uniform for $-\infty<x< \infty$ (see  \cite[Ch. 10, \S7]{Olver}).

Now the function $\chi(x,\lambda)$, defined in \eqref{chi-3}, satisifes
\begin{align*}
\lim_{x\to\infty}\chi(x,\lambda) &=0,\\
\chi(0,\lambda) &=\log a(\lambda)-a_{0}(\lambda),
\end{align*}
where
$$a_{0}(\lambda)=\sqrt{\nu^{2}+1}-\nu\log\left(\nu+\sqrt{\nu^{2}+1}\right)-\frac{1}{4}\log(\nu^{2}+1)+\frac{1}{2}\log\frac{\pi}{2}-\log K_{0}(1).$$
The function $a_{0}(\lambda)$ admits an asymptotic expansion \eqref{alpha-0-0}, where
$$\al_{0}(\lambda)=\nu\log(2\nu)-\nu-\frac{1}{2}\log\nu+\frac{1}{2}\log\frac{\pi}{2}-\log K_{0}(1).$$
This formula agrees with the asymptotic 
\begin{equation} \label{K-nu}
K_{\nu}(e^{x})=\sqrt{\frac{\pi}{2\nu}}\left(\frac{e^{x+1}}{2\nu}\right)^{\!-\nu}(1+O(|\nu|^{-1}))
\end{equation}
for fixed $x$ and $\nu\to\infty$, and with the Weyl's law.

The function $\sigma(x,\lambda)$, defined in \eqref{sigma-1}, satisfies the Riccati equation \eqref{Ric}. Its asymptotic expansion \eqref{c-series} can be easily obtained by using the substitution
$\tau(t,\lambda)=\sigma(t+\log\nu,\lambda)$, which transforms \eqref{Ric} into the equation
\begin{equation}\label{Ric-2}
\tau'=-\tau^{2}+2\nu\sqrt{1+e^{2t}}\tau +\frac{e^{2t}}{1+e^{2t}}\tau+\frac{4e^{2t}-e^{4t}}{4(1+e^{2t})^{2}}.
\end{equation}
We have
$$\tau(t,\nu)=\sum_{n=1}^{\infty}\frac{c_{n}(t)}{\nu^{n}},$$
where
$$c_{1}(t)=\frac{e^{4t}-4e^{2t}}{8(1+e^{2t})^{5/2}}$$
and recursively
\begin{equation*}
2\sqrt{1+e^{2t}}c_{n+1}(t)=c_{n}'(t)-\frac{e^{2t}}{1+e^{2t}}c_{n}(t)+\sum_{k=1}^{n}c_{k}(t)c_{n-k}(t).
\end{equation*}
Plugging this into \eqref{a-asym-2}, we obtain
$$\log a(\lambda)=\nu\log(2\nu)-\nu-\frac{1}{2}\log\nu+\frac{1}{2}\log\frac{\pi}{2}+\sum_{n=1}^{\infty}\frac{\alpha_{n}-c_{n}(\nu)}{\nu^{n}} +O(|\nu|^{-\infty}),$$
where
$$c_{n}(\nu)=-\int_{-\log\nu}^{\infty}\sigma_{n}(t)dt.$$
Moreover, it can be shown (cf. \cite[Ch.10, \S7]{Olver}) that the coefficients $c_{n}(\nu)$ are polynomials in $\nu/\sqrt{\nu^{2}+1}$, which gives the asymptotic expansion of $\log a(\lambda)$. 

\subsection{Potential $q(x)=2\cosh 2x$ on the line}\label{cosh-line-0}
Here we consider the modified Mathieu differential equation
\begin{equation} \label{S-cosh}
-\psi''+2\cosh 2x\,\psi=\lambda\psi,\quad -\infty<x<\infty.
\end{equation}

It is well-known that the antiderivative of the function
$\sqrt{2\cosh 2x-\lambda}$ is expressed in terms of elliptic integrals. Namely,  let
$$F(\vp,k)=\int_{0}^{\vp}\frac{d\theta}{\sqrt{1-k^{2}\sin^{2}\theta}}\quad\text{and}\quad E(\vp,k)=\int_{0}^{\vp}\sqrt{1-k^{2}\sin^{2}\theta}\,d\theta$$
be, correspondingly, elliptic integrals of the first and second kinds, where $0<k^{2}<1$ and $-\pi/2<\vp<\pi/2$. Putting $\lambda=2-\nu^{2}$, where $\nu>0$, we have for $\nu>2$ (see \cite[formula (22) on p. 136]{GR})
\begin{gather}
\int_{0}^{x}\sqrt{\nu^{2}-2+2\cosh 2s}\,ds \nonumber\\ 
=\nu(F(\vp,k)-E(\vp,k))+\tanh\frac{x}{2}\,\sqrt{\nu^{2}-2+2\cosh 2x}, \label{F-E}
\end{gather}
where
\begin{equation}\label{k-phi}
\vp=\sin^{-1}\left(\tanh\frac{x}{2}\right)\quad\text{and}\quad k^{2}=\frac{\nu^{2}-4}{\nu^{2}}.
\end{equation}
The corresponding solutions $\psi_{1}(x,\lambda)$  and $\psi_{2}(x,\lambda)$ of the differential equation \eqref{S-cosh} with asymptotics \eqref{chi-1}--\eqref{chi-2}
are
\begin{align}
\psi_{1}(x,\lambda) & \simeq\frac{C_{1}(\lambda)}{\sqrt[4]{2\cosh 2x +\nu^{2}}}e^{\nu(F(\vp,k)-E(\vp,k))+\tanh\frac{x}{2}\,\sqrt{\nu^{2}-2+2\cosh 2x}},\label{1}\\
\psi_{2}(x,\lambda) & \simeq \frac{C_{2}(\lambda)}{\sqrt[4]{2\cosh 2x +\nu^{2}}}e^{-\nu(F(\vp,k)-E(\vp,k))-\tanh\frac{x}{2}\,\sqrt{\nu^{2}-2+2\cosh 2x}},\label{2}
\end{align}
where $\lambda=2-\nu^{2}\to-\infty$ (cf. \cite{Grig}).

Since $\vp\to\pi/2$ as $x\to\infty$, we have
$$\int_{0}^{x}\sqrt{\nu^{2}-2+2\cosh 2s}\,ds=\nu(\bm{K}(k)-\bm{E}(k)))+e^{x}+O(e^{-x}),$$
where $\bm{K}(k)$ and $\bm{E}(k)$ are, respectively, complete elliptic integrals of the first and second kinds.
Choosing the constants in \eqref{1}--\eqref{2} as
\begin{equation}\label{C-1-2}
C_{1}(\lambda)=\sqrt\frac{\pi}{2}\,e^{-\nu(\bm{K}(k)-\bm{E}(k))}\quad\text{and}\quad C_{2}(\lambda)=\sqrt\frac{1}{2\pi}\,e^{\nu(\bm{K}(k)-\bm{E}(k))},
\end{equation}
we get a solution $\psi_{1}(x,\lambda)$ with the same asymptotic as $x\to\infty$ as the modified Bessel function of the second kind $K_{i\sqrt\lambda}(e^{x})$,
$$\psi_{1}(x,\lambda)=\sqrt\frac{\pi}{2e^{x}}\,e^{-e^{x}}(1+O(e^{-x}))\quad\text{as}\quad x\to\infty,$$
and a solution $\psi_{2}(x,\lambda)$ with the same asymptotic as $x\to\infty$ as 
the modified Bessel function of the first kind $I_{i\sqrt\lambda}(e^{x})$, 
$$\psi_{2}(x,\lambda)=\sqrt\frac{1}{2\pi e^{x}}\,e^{e^{x}}(1+O(e^{-x}))\quad\text{as}\quad x\to\infty.$$

We have $W(\psi_{1},\psi_{2})=1$ and 
\begin{equation}\label{-+-2}
\psi_{1}(x,\lambda)=t_{11}(\lambda)\psi_{1}(-x,\lambda)+ t_{12}(\lambda)\psi_{2}(-x,\lambda),
\end{equation}
where
\begin{equation}\label{trans-2}
t_{12}(\lambda)=W(\psi_{1}(x,\lambda),\psi_{1}(-x,\lambda))
\end{equation}
is an entire function of order $1/2$ with zeros --- the eigenvalues $\lambda_{n}$ of the Schr\"{o}dinger operator \eqref{H-M}. Corresponding Fredholm determinant is
$$a(\lambda)=\frac{t_{12}(\lambda)}{t_{12}(0)}.$$

Let $\chi(x,\lambda)$ be the function defined in \eqref{chi-3}.
Since the antiderivative \eqref{F-E}--\eqref{k-phi} is an odd function of $x$, as in Sect. \ref{line} we get
\begin{align}
\lim_{x\to\infty}\chi(x,\lambda) & = 0,\label{a-1}\\
\lim_{x\to-\infty}\chi(x,\lambda) &= \log a(\lambda)+\log\pi+\log t_{12}(0)-2\nu(\bm{K}(k)-\bm{E}(k)).\label{a-2}
\end{align}

We state the main result of this section.
\begin{thm} \label{main-thm} Fredholm determinant $a(\lambda)$ of the Schr\"{o}dinger operator \eqref{H-M} admits the following asymptotic expansion as $\lambda=2-\nu^{2}\to-\infty$
 \begin{equation} \label{a-as-cosh-0}
\log a(\lambda)=2\nu(\bm{K}(k)-\bm{E}(k))-\log\pi-\log t_{12}(0)+\sum_{n=1}^{\infty}\frac{c_{n}}{\nu^{n}} +O(|\nu|^{-\infty}),
\end{equation}
where $k^{2}=1-4\nu^{-2}$ and the coefficients $c_{n}$ are determined explicitly.
\end{thm} 

\begin{proof} 
It follows from \eqref{a-1}--\eqref{a-2} that 
\begin{equation}\label{a-main}
\log a(\lambda) =2\nu(\bm{K}(k)-\bm{E}(k))-\log\pi-\log t_{12}(0)-\int_{-\infty}^{\infty}\sigma(x,\lambda)dx,
\end{equation}
where $\sigma(x,\lambda)$ is defined in \eqref{sigma-1} and satisfies the Riccati equation \eqref{Ric}. Using 
$$2\cosh 2x-\lambda=\nu^{2}+ 2\cosh 2x-2=\nu^{2}+4\sinh^{2}x,$$
we can rewrite the equation \eqref{Ric} as
\begin{gather*}
\frac{d\sigma}{dx}=-\sigma^{2}+2\sqrt{\nu^{2}+4\sinh^{2}x}\,\sigma +\frac{2\sinh 2x}{\nu^{2}+4\sinh^{2}x}\sigma\\
+\frac{8\cosh 2x}{\nu^{2}+4\sinh^{2}x}-5\left(\frac{\sinh 2x}{{\nu^{2}+4\sinh^{2}x}}\right)^{2}.
\end{gather*}
It is convenient to change variables by
$\sinh x=\dfrac{\nu}{2}\sinh y$,  
so
$$\nu^{2}+4\sinh^{2}x=\nu^{2}\cosh^{2}y$$ 
and
\begin{gather*}
\frac{dx}{dy}=\frac{\nu\cosh y}{2\cosh x}=\frac{\nu\cosh y}{\sqrt{4+\nu^{2}\sinh^{2}y}}
=\left(1-\frac{k^{2}}{\cosh^{2}y}\right)^{\!\!-\frac{1}{2}}.
\end{gather*}
Next, introduce the function $\tau(y,\lambda)=\sigma\left(\sinh^{-1}(\frac{\nu}{2}\sinh y), \lambda\right)$, so
\begin{equation}\label{sigma-tau}
\int_{-\infty}^{\infty}\sigma(x,\lambda)dx=\int_{-\infty}^{\infty}\tau(y,\lambda)\left(1-\frac{k^{2}}{\cosh^{2}y}\right)^{\!\!-\frac{1}{2}}dy.
\end{equation}
The Riccati equation for  $\tau(y,\lambda)$ takes the form
\begin{gather}
\left(1-\frac{k^{2}}{\cosh^{2}y}\right)^{\!\!\frac{1}{2}}\frac{d\tau}{dy}=-\tau^{2}+2\nu\cosh y\,\tau +\tanh y\left(1-\frac{k^{2}}{\cosh^{2}y}\right)^{\!\!\frac{1}{2}}\tau \nonumber\\
+4\tanh^{2}y-\frac{5}{4}\tanh^{4}y+\frac{1}{\nu^{2}\cosh^{2}y}(8-\frac{5}{4}\tanh^{2}y).\label{Ric-tau-2}
\end{gather}
Now it is straightforward to show that that equation \eqref{Ric-tau-2} 
admits an asymptotic solution 
$$\tau(y,\lambda)=\sum_{n=1}^{\infty}\frac{\tau_{n}(y)}{\nu^{n}},$$
where 
$$\tau_{1}(y)=-\frac{1}{2\cosh y}(4\tanh^{2}y-\frac{5}{4}\tanh^{4}y).$$
The coefficients $\tau_{n}(y)$ are obtained recursively using the expansion
$$\left(1-\frac{k^{2}}{\cosh^{2}y}\right)^{\!\!\frac{1}{2}} =\sum_{m=1}^{\infty}(-1)^{m}\binom{\frac{1}{2}}{m}\frac{k^{2m}}{\cosh^{2m}y}$$
and the binomial expansion for $k^{2m}=(1-4\nu^{-2})^{m}$; only finitely many terms from these expansions are needed in order to get the $n$-th term. Substituting the asymptotic expansion for $\tau(y,\lambda)$ into \eqref{sigma-tau} and \eqref{a-main} and using 
$$\left(1-\frac{k^{2}}{\cosh^{2}y}\right)^{\!\!-\frac{1}{2}} =\sum_{m=1}^{\infty}(-1)^{m}\binom{-\frac{1}{2}}{m}\frac{k^{2m}}{\cosh^{2m}y},$$
we get the desired asymptotic expansion \eqref{a-as-cosh-0}. 
\end{proof}
\begin{rem}
Since
$$\bm{K}(k)=\log\frac{4}{\sqrt{1-k^{2}}}+o(1)\quad\text{and}\quad \bm{E}(k)=1+o(1)\quad\text{as}\quad k\to 1,$$
we have as $\nu\to\infty$,
\begin{equation*}
\nu(\bm{K}(k)-\bm{E}(k))=\nu\log(2\nu)-2\nu +o(1).
\end{equation*}
Here the $o(1)$ term --- the remainder --- is of the form $f_{1}(\nu^{-1})+f_{2}(\nu^{-1})\log\nu$, where
$f_{1}(x)$ and $f_{2}(x)$ are convergent for $|x|<1$ power series in $x^{2}$ that vanish at $x=0$
(see \cite[p. 54]{Cayley}).
\end{rem}
\section{Existence of the solution $\psi_{1}(x,\lambda)$} \label{E-value}
 
Here for the convenience of the reader we prove that differential equation \eqref{S-cosh} has a solution $\psi_{1}(x,\lambda)$ which asymptotically as $x\to\infty$ behaves like the modified Bessel function of the second kind.
Namely, we have the following result.
\begin{thm} Modified Mathieu equation \eqref{S-cosh} has a solution $\psi_{1}(x,\lambda)$ with the following asymptotic
$$\psi_{+}(x,\lambda) =K_{ik}(e^{x})(1+o(1)) \quad\text{as}\quad x\to\infty,\quad\text{where}\quad \lambda=k^{2}.$$
For fixed $x$ the function  $\psi_{1}(x,\lambda)$ is entire of order $1/2$.
\end{thm}
\begin{proof} We consider the Schr\"{o}dinger operator $H=-\dfrac{d^{2}}{dx^{2}}+2\cosh 2x$ as a perturbation  of the operator $H_{0}=-\dfrac{d^{2}}{dx^{2}}+e^{2x}$ by a small as $x\to\infty$ potential $e^{-2x}$. The operator $H_{0}$ has a simple absolutely continuous spectrum $[0,\infty)$, and for $\lambda\in\mathbb{C}\setminus [0,\infty)$ its resolvent $R_{\lambda}^{0}=(H_{0}-\lambda I)^{-1}$ is an integral operator in $L^{2}(\RR)$ with the integral kernel (see, e.g., \cite[Ch. 4.15]{Titchmarsh})
\begin{equation}\label{Res-0}
R_{\lambda}^{0}(x,y)=\begin{cases} I_{-ik}(e^{x})K_{ik}(e^{y}), &\quad x\leq y,\\
K_{ik}(e^{x})I_{-ik}(e^{y}),&\quad x\geq y.\end{cases}
\end{equation}
Here $k=\sqrt\lambda$ and $\mathrm{Im}\,k>0$. The operator $H_{0}$ has a Volterra type Green's function
$$G(x,y,k)= \begin{cases} I_{-ik}(e^{x})K_{ik}(e^{y})-K_{ik}(e^{x})I_{-ik}(e^{y}),&\quad x\leq y,\\
0, &\quad x>y.\end{cases}$$
Since 
$$K_{\nu}(z)=\frac{\pi}{2\sin\pi\nu}(I_{-\nu}(z)-I_{\nu}(z)),$$
$G(x,y,k)$ is an even function of $k$ and, therefore, is an entire function of $\lambda$ of order $1/2$.

The function $\psi_{1}(x,\lambda)$ satisfies the integral equation 
\begin{equation}\label{Qol}
\psi_{1}(x,\lambda)=K_{ik}(e^{x})+\int_{x}^{\infty}G(x,y,k)e^{-2y}\psi_{1}(y,\lambda)dy,
\end{equation}
which can be solved by successive approximations.
Indeed, put  $f_{0}(x,k)=K_{ik}(e^{x})$ and
$$f_{n}(x,k)=\int_{x}^{\infty}G(x,y,k)e^{-2y}f_{n-1}(y,k)dy.$$
Using the estimates
\begin{equation}\label{est-0}
|K_{ik}(e^{x})|\leq Ce^{-e^{x}}e^{-\frac{1}{2}x}\quad\text{and}\quad |I_{ik}(e^{x})|\leq Ce^{e^{x}}e^{-\frac{1}{2}x}
\end{equation}
for $a\leq x<\infty$ (the constant $C$ depends on $a$), we have $f_{n}(x,-k)=f_{n}(x,k)$ and
\begin{equation}\label{est}
|f_{n}(x,k)|\leq \frac{2^{n}C^{2n+1}}{3^{n} n!}e^{-e^{x}}e^{-\frac{6n+1}{2}x},
\end{equation}
which can be easily proved by induction. Namely, \eqref{est} holds for $n=0$, and using the estimates \eqref{est-0}, we get
\begin{gather*}
|f_{n+1}(x,k)|\leq |I_{-ik}(e^{x})|\int_{x}^{\infty}e^{-2y}|K_{ik}(e^{y})f_{n}(y,k)|dy \\
+|K_{ik}(e^{x})|\int_{x}^{\infty}e^{-2y}|I_{-ik}(e^{y})f_{n}(y,k)|dy\\
\leq \frac{2^{n}C^{2n+3}}{3^{n} n!} \left[e^{e^{x}}e^{-\frac{1}{2}x}\int_{x}^{\infty}e^{-2y}e^{-2e^{y}}e^{-\frac{1}{2}y}e^{-\frac{6n+1}{2}y}dy\right.\\
\left.+e^{-e^{x}}e^{-\frac{1}{2}x}\int_{x}^{\infty}e^{e^{y}}e^{-\frac{1}{2}y}e^{-2y}e^{-e^{y}}e^{-\frac{6n+1}{2}y}dy\right]\\
\leq \frac{2^{n+1}C^{2n+3}}{3^{n+1} (n+1)!}e^{-e^{x}}e^{-\frac{6n+7}{2}x}.
\end{gather*}
Thus
$$\psi_{1}(x,\lambda)=\sum_{n=0}^{\infty}f_{n}(x,k)$$
is given by the absolutely convergent series and
$$|\psi_{1}(x,\lambda)-K_{ik}(e^{x})|\leq Ce^{-e^{x}}e^{-\frac{7}{2}x}.\qedhere$$
\end{proof}

As in Sect. \ref{cosh-line-0},  we have
$$t_{12}(\lambda)=W(\psi_{1}(x,\lambda),\psi_{1}(-x,\lambda))=-2\psi_{1}(0,\lambda)\psi'_{1}(0,\lambda),$$
and the eigenvalue problem \eqref{S-cosh}  is equivalent to two radial eigenvalue problems 
\begin{alignat*}{2}
-\psi'' +2\cosh 2x\,\psi & =k^{2}\psi, \quad 0<x<\infty, &\quad \psi(0)=0,\\
-\psi'' +2\cosh 2x\,\psi& =k^{2}\psi, \quad 0<x<\infty, &\quad \psi'(0)=0.
\end{alignat*}

\begin{acknowledgments}
It is a pleasure to thank Ari Laptev for the discussion of the trace identities and for drawing my attention to the reference \cite{P}. I am also grateful to Marcos Mari\~{n}o for the useful comments and for pointing to the references 
\cite{Mar2} and \cite{Voros1, Voros2}.
\end{acknowledgments}

\end{document}